\documentclass{elsarticle}
\usepackage{graphicx}
\usepackage{multirow}
\usepackage{amsmath,amssymb,amsfonts}
\usepackage{amsthm}
\usepackage{mathrsfs}
\usepackage[title]{appendix}
\usepackage{xcolor}
\usepackage{textcomp}
\usepackage{manyfoot}
\usepackage{booktabs}
\usepackage{algorithm}
\usepackage{algorithmicx}
\usepackage{algpseudocode}
\usepackage{listings}
\usepackage{geometry}
\usepackage{indentfirst}
\usepackage{cases}
\usepackage{tikz}
\usepackage{dsfont} 
\usepackage{float}
\usepackage{lmodern}
\usepackage{geometry}
\usepackage{indentfirst}
\usepackage{cases}
\newtheorem{thm}{Theorem}[section]
\newtheorem{lem}[thm]{Lemma}

\numberwithin{equation}{section}
\begin{document}
\begin{frontmatter}
\title{Global well-posedness for dissipative IPM with data close to a class of special solutions}
\author[author1]{Liangchen Zou}
\address[author1]{School of Mathematical Sciences, University of Science and Technology of China, Hefei, Anhui, 230026, PR China, zlc0601@mail.ustc.edu.cn}

\date{\today}
\begin{abstract}
In this paper, we consider the 2-D dissipative incompressible porous media (IPM) equation in both supercritical and subcritical cases. The dissipative IPM equation admits a class of special solutions of the form $\rho(x_1,x_2,t)=f(x_2,t)$, which decay in the mode of the 1-D fractional heat equation. Our main result is the global well-posedness for the dissipative IPM equation with initial data close to this class of special solutions provided that $f$ satisfies certain regularity assumptions.
\end{abstract}
\end{frontmatter}
\section{Introduction}
Let $0<\alpha<2$. We consider the dissipative incompressible porous media (IPM) equation
\begin{equation}\begin{aligned}
\partial_t\rho+u\cdot\nabla\rho=-\Lambda^\alpha\rho,
\end{aligned}\label{1.1}\end{equation}
with a velocity field $u$ satisfying the momentum equation given by Darcy's law:
\begin{equation}\frac{\mu}{\kappa}u=-\nabla p -g(0,\rho),\label{1.2}\end{equation}
where $x=(x_1,x_2)\in\mathbb{R}^2$, $t>0$, $\Lambda=\sqrt{-\Delta}$, $u=(u_1,u_2)$ is the incompressible velocity ($\nabla\cdot u=0$), $p$ is the pressure, $\rho$ is the density, $\mu$ is the dynamic viscosity, $\kappa$ is the permeability, and $g$ is the gravity acceleration. For more details and physical background, we refer to \cite{nield2006convection}. For convenience, we set $\mu=\kappa=g=1$, so that we solve from (\ref{1.2}) that
$$u_1=-R_1R_2\rho,\;\quad u_2=R_1^2\rho,$$
where $R_i:=\partial_i\Lambda^{-1}$, $i\in\left\{1,2\right\}$ is the $i$th Riesz transform on $\mathbb{R}^2$.\\
\indent There have been many results on the well-posedness of IPM equations in different settings, including the whole space case \cite{cordoba2007analytical}, in a strip domain \cite{castro2019global}, patch-type solution \cite{cordoba2007contour}, patch-type solution for singular IPM \cite{friedlander2012singular}. In addition, other results include the lack of uniqueness for weak solutions \cite{cordoba2011lack, szekelyhidi2012relaxation}, instability \cite{castro2013breakdown, kiselev2023small}, and long-time behavior \cite{castro2019global, cordoba2007contour, elgindi2017asymptotic}. For the dissipative IPM, local existence for arbitrary data and global existence for small data are obtained in the supercritical case \cite{xue2009well} and critical case \cite{yuan2009global}.\\
\indent As mentioned by Elgindi \cite{elgindi2017asymptotic}, the IPM equation admits a class of steady solutions $\rho=f(x_2)$, where $f$ can be any $C^1$ function. As to the dissipative IPM equation, $\rho=f(x_2)$ is no longer a steady solution since $\Lambda^{\alpha}\rho\neq0$. However, if we let $f$ be the solution of the 1-D fractional heat equation $\partial_tf=-\Lambda^{\alpha}f$, then $\rho(x_1,x_2,t):=f(x_2,t)$ is a solution of (\ref{1.1}) with $u=0$. In contrast, Bulut and Dong \cite{bulut2024global} observe that any radially symmetric function is a steady solution of the surface quasi-geostrophic (SQG) equation, and any radially symmetric solution to the 2-D fractional heat equation is also a solution to the dissipative SQG equation.\\
\indent The authors of \cite{bulut2024global} proved global well-posedness for the SQG equation with initial data which is a small perturbation of a radial function. Motivated by this, in this paper we consider the global well-posedness for the dissipative IPM equation with initial data which is a perturbation of a function dependent only on $x_2$. The essential difference between the SQG equation and the IPM equation is that radial solutions of the SQG equation can lie in a certain Sobolev space $H^s(\mathbb{R}^2)$, while a function of only $x_2$ can never lie in any Sobolev space on $\mathbb{R}^2$. Moreover, the radial solution to the dissipative SQG equation should satisfy the 2-D fractional heat equation, while to make $\rho(x_1,x_2,t)=f(x_2,t)$ a solution of the dissipative IPM equation, one needs $f$ to satisfy the 1-D fractional heat equation. The difference in dimension results in different decay rates of solutions to heat equations and thus different controls over the nonlinear terms.\\
\indent Now we state the main theorem. To avoid ambiguity, let $\tilde{\Lambda}:=\sqrt{-\partial_x^2}$ be the operator acting on functions defined in the 1-D space $\mathbb{R}$ with domain $\dot{H}^1(\mathbb{R})$. 
\begin{thm}
Suppose one of the following holds:\\
(\romannumeral1) $0<\alpha<1$, $s=2-\alpha$, $\nabla f\in L^{\infty}(\mathbb{R)}$, $\tilde{\Lambda}^{1+s-\frac{\alpha}{2}} f\in L^{p}(\mathbb{R})$ with $p\in[2,\infty)$, and $\nabla f\in L^q(\mathbb{R})$ for a given $q\in\left[1,\frac{1}{\alpha}\right)$.\\
(\romannumeral2) $1\leq\alpha<2$, $s\geq1$, $\nabla f\in L^{\infty}(\mathbb{R)}$, $\tilde{\Lambda}^{1+s-\frac{\alpha}{2}} f\in L^{p}(\mathbb{R})$ with $p\in[2,\infty)$, and $f\in L^q(\mathbb{R})$ for a given $q\in\left[1.\frac{1}{\alpha-1}\right)\cap[1,p]$.\\
Then there exists $\varepsilon>0$ such that if $g\in H^s(\mathbb{R}^2)$ satisfies $\|g\|_{H^s}\leq\varepsilon,$ then the initial value problem
\begin{equation}\left\{\begin{aligned} 
&\partial_t\rho+u\cdot\nabla\rho=-\Lambda^\alpha \rho\\
&u=(u_1,u_2)=(-R_1R_2\rho, R_1^2\rho)\\
&\rho(x,0)=f(x_2)+g(x_1,x_2)
\end{aligned}\right.\label{1.3}\end{equation}
has a unique global solution in $C\left([0,\infty);H^{s}(\mathbb{R}^2)\right)$.
\label{thm1.1}\end{thm}
\noindent\textbf{Acknowledgements.} The author is grateful for the helpful discussions with Professor \'{A}ngel Castro (Instituto de Ciencias Matemáticas, ICMAT-CSIC-UAM-UCM-UC3M). The author acknowledges support from the China Scholarship Council Program (Project ID 202306340142).
\section{Decomposition and a priori estimate}
Let $K_\alpha(t):=\exp\left\{-t\tilde{\Lambda}^\alpha \right\}$ be the semigroup generated by the operator $-\tilde{\Lambda}^\alpha$ on $\mathbb{R}$ for $t>0$ so that
$\rho_0(x_2,t):=[K_\gamma(t)f](x_2)$, $(x_2,t)\in\mathbb{R}\times\mathbb{R}^+$ is the solution to 
\begin{equation}\left\{\begin{aligned}
&\partial_t\rho_0=-\tilde{\Lambda}^\alpha\rho_0\\
&\rho_0(x_2,0)=f(x_2).
\end{aligned}\right.\label{2.1}\end{equation}
Suppose that $\rho$ is a solution to (\ref{1.3}) with the decomposition $\rho=\rho_1+\rho_0$, then $\rho_1$ must solve the equation
\begin{equation}
\partial_t\rho_1+u\cdot\nabla\rho_1+u_2\partial_{x_2}\rho_0=-\Lambda^{\alpha}\rho_1
\label{2.2}\end{equation}
with $\rho_1(x,0)=g(x)$. In addition, since $\rho_0$ is a function of only $x_2$ and independent of $x_1$, the velocity satisfies
$$u_1=-R_1R_2\rho_1,\quad u_2=R_1^2\rho_1,$$
and consequently $\|u\|_{\dot{H}^\gamma}=\|R_1\rho_1\|_{\dot{H}^\gamma}\leq\|\rho_1\|_{\dot{H}^\gamma}$, $\gamma\geq0$.
Let $\Lambda_1:=\sqrt{-\partial_{x_1}^2}$, and $\Lambda_2:=\sqrt{-\partial_{x_2}^2}$ with domain $\dot{H}^1(\mathbb{R}^2)$. We now show some properties of the operators $\Lambda_i$ , $i\in\left\{1,2\right\}$ which we will use later.
\begin{lem} Suppose $s\geq0$, then the following properties hold:\\
(\romannumeral1) $(\Lambda_1^s+\Lambda_2^s)^{-1}\Lambda^s$ is bounded and invertible on $L^2(\mathbb{R}^2)$.\\
(\romannumeral2) Suppose $f\in \dot{H}^s(\mathbb{R}^2)$, then for almost every $x_1\in\mathbb{R}$
$$(\Lambda_2^sf)(x_1,\cdot)=\tilde{\Lambda}(f(x_1,\cdot)).$$
(\romannumeral3) Let $f\in\dot{H}^s(\mathbb{R}^2)$, $g\in L^1_{loc}(\mathbb{R)}$ be regular enough and $h(x_1,x_2):=f(x_1,x_2)g(x_2)$. Then
$$\Lambda_1^sh=g\Lambda_1^sf.$$
\label{lem2.1}\end{lem}
\begin{proof}
The assertion (\romannumeral1) holds since the operator $(\Lambda^s_1+\Lambda^s_2)^{-1}\Lambda^s$ is the multiplier with the symbol $\left(|\xi_1|^s+|\xi_2|^s\right)^{-1}|\xi|^s$, which is bounded from above and below by constants. As a consequence, for any $f\in\dot{H}^s(\mathbb{R}^2)$, $\|f\|_{\dot{H}^s}\simeq\|\Lambda_1^sf\|_{L^2}+\|\Lambda_2^sf\|_{L^2}$. To prove (\romannumeral2), we denote the Fourier transform in the variable $x_i$ by $\mathcal{F}_{i}$, $i\in\left\{1,2\right\}.$ Since $f\in \dot{H}^s(\mathbb{R})$, $f(x_1,\cdot)\in \dot{H}^s(\mathbb{R})$ for almost every $x_1\in\mathbb{R}$. Hence
$$\begin{aligned}
(\Lambda_2^sf)(x_1,x_2)=&\mathcal{F}_1^{-1}\mathcal{F}_2^{-1}\left(|\xi_2|^s\mathcal{F}_1\mathcal{F}_2f\right)(x_1,x_2)\\
=&\mathcal{F}_2^{-1}\mathcal{F}_1^{-1}\left(\mathcal{F}_1\left(|\xi_2|^s\mathcal{F}_2f\right)\right)(x_1,x_2)\\
=&\mathcal{F}_2^{-1}\left(|\xi_2|^s\mathcal{F}_2f\right)(x_1,x_2)\\
=&\tilde{\Lambda}(f(x_1,\cdot))(x_2).
\end{aligned}$$
For (\romannumeral3), by noting that $\mathcal{F}_1h=g\mathcal{F}_1f$, it follows
$$\begin{aligned}
\Lambda_1^sh=&\mathcal{F}_1^{-1}\mathcal{F}_2^{-1}\left(|\xi_1|^s\mathcal{F}_1\mathcal{F}_2h\right)\\
=&\mathcal{F}_1^{-1}\mathcal{F}_2^{-1}\left(\mathcal{F}_2(|\xi_1|^s\mathcal{F}_1h)\right)\\
=&\mathcal{F}_1^{-1}\left(|\xi_1|^sg\mathcal{F}_1f\right)\\
=&g\Lambda_1^sf.
\end{aligned}$$
\end{proof}
We now consider a priori bounds on solutions to (\ref{2.2}).
\begin{lem}
Let $\rho_0$, $\rho_1$ satisfy (\ref{2.1})(\ref{2.2}). Then for all $t>0$, the following estimates hold.\\
(\romannumeral1) If $0\leq\alpha<1$, then
\begin{equation}\begin{aligned}
\left|\int_{\mathbb{R}^2}\rho_1u\cdot\nabla\rho_0dx\right|\lesssim\|\Lambda^\frac{\alpha}{2}\rho_1\|_{L^2}^2\|\nabla\rho_0\|_{L^\frac{1}{\alpha}_{x_2}}.
\end{aligned}\label{2.3}\end{equation}
\noindent For any $0\leq\alpha\leq2$,
\begin{equation}\begin{aligned}
\left|\int_{\mathbb{R}^2}\rho_1u\cdot\nabla\rho_0dx\right|\leq\|\rho_1\|_{L^2}^2\|\nabla\rho_0\|_{L^\infty_{x_2}}
\end{aligned}\label{2.4}\end{equation}
\label{lem2.2}\end{lem}
\noindent(\romannumeral2) If $0<\alpha\leq1$, then for $s=2-\alpha$
\begin{equation}\begin{aligned}
\left|\int_{\mathbb{R}^2}\Lambda^s\rho_1\Lambda^s(u\cdot\nabla\rho_1)dx\right|\lesssim\|\Lambda^s\rho_1\|_{L^2}\|\Lambda^{s+\frac{\alpha}{2}}\rho_1\|_{L^2}^2.    
\end{aligned}\label{2.5}\end{equation}
For $\alpha\in[1,2]$ and $s\geq1$
\begin{equation}\begin{aligned}
\left|\int_{\mathbb{R}^2}\Lambda^s\rho_1\Lambda^s(u\cdot\nabla\rho_1)dx\right|\lesssim\|\Lambda^s\rho_1\|_{L^2}\left(\|\Lambda^\frac{\alpha}{2}\rho_1\|_{L^2}^2+\|\Lambda^{s+\frac{\alpha}{2}}\rho_1\|_{L^2}^2\right).
\end{aligned}\label{2.6}\end{equation}
(\romannumeral3) If $0<\alpha<1$, then for $s=2-\alpha$
\begin{equation}\begin{aligned}
\left|\int_{\mathbb{R}^2}\Lambda^s\rho_1\Lambda^s(u\cdot\nabla\rho_0)dx\right|\lesssim\|\Lambda^{s+\frac{\alpha}{2}}\rho_1\|_{L^2}^2\|\nabla\rho_0\|_{L^\frac{1}{\alpha}_{x_2}}+\|\Lambda^{s+\frac{\alpha}{2}}\rho_1\|_{L^2}\|\Lambda^\frac{\alpha}{2}\rho_1\|_{L^2}\|\nabla\tilde{\Lambda}^s\rho_0\|_{L^\frac{1}{\alpha}_{x_2}}.
\end{aligned}\label{2.7}\end{equation}
For any $0\leq\alpha\leq2$, $s\geq1$ and $2\leq p<\infty$
\begin{equation}\begin{aligned}
\left|\int_{\mathbb{R}^2}\Lambda^s\rho_1\Lambda^s(u\cdot\nabla\rho_0)dx\right|\lesssim\|\Lambda^{s+\frac{\alpha}{2}}\rho_1\|_{L^2}\|u\|_{H^s}\left(\|\nabla\rho_0\|_{L^\infty_{x_2}}+\|\nabla\tilde{\Lambda}^{s-\frac{\alpha}{2}}\rho_0\|_{L^p_{x_2}}\right).
\end{aligned}\label{2.8}\end{equation}
\begin{proof}
First, we note that $\nabla\rho_0$ is a function of only $x_2$. By the Sobolev embedding, it holds for $\alpha\in[0,1)$ that
$$\begin{aligned}
\left|\int_{\mathbb{R}^2}\rho_1u\cdot\nabla\rho_0dx\right|\lesssim&\int_{\mathbb{R}}\|\rho_1\|_{L_{x_2}^{\frac{2}{1-\alpha}}}\|u\|_{L_{x_2}^\frac{2}{1-\alpha}}\|\nabla\rho_0\|_{L_{x_2}^\frac{1}{\alpha}}dx_1\\
\lesssim&\int_{\mathbb{R}}\|\Lambda_2^\frac{\alpha}{2}\rho_1\|_{L^2_{x_2}}\|\Lambda_2^\frac{\alpha}{2}u\|_{L^2_{x_2}}dx_1\|\nabla\rho_0\|_{L^\frac{1}{\alpha}_{x_2}}\\
\lesssim&\|\Lambda_2^\frac{\alpha}{2}\rho_1\|_{L^2}\|\Lambda_2^\frac{\alpha}{2}u\|_{L^2}\|\nabla\rho_0\|_{L^\frac{1}{\alpha}_{x_2}}.
\end{aligned}$$
Hence we obtain (\ref{2.3}) in view of $\|\Lambda_2^\frac{\alpha}{2}u\|_{L^2}\leq\|\Lambda_2^\frac{\alpha}{2}\rho_1\|_{L^2}\leq\|\Lambda^\frac{\alpha}{2}\rho_1\|_{L^2}$. Meanwhile, since $\|u\|_{L^2}\leq\|\rho_1\|_{L^2}$, we obtain (\ref{2.4}) immediately for any $\alpha\in[0,2]$. Suppose now $0<\alpha\leq1$ and $s=2-\alpha$. Since $\nabla\cdot u=0$, we have 
$\int_{\mathbb{R}}\Lambda^s\rho_1(u\cdot\nabla\Lambda^s\rho_1)dx=0$, and thus
$$\begin{aligned}
\left|\int_{\mathbb{R}^2}\Lambda^s\rho_1\Lambda^s(u\cdot\nabla\rho_1)dx\right|=&\left|\int_{\mathbb{R}^2}\Lambda^s\rho_1\left(\Lambda^s(u\cdot\nabla\rho)-u\cdot\nabla\Lambda^s\rho_1\right)dx\right|\\
\lesssim&\|\Lambda^s\rho_1\|_{L^2}\left(\|\Lambda^su\|_{L^{p_1}}\|\nabla\rho\|_{L^{p_2}}+\|\Lambda^{s}\rho_1\|_{L^{p_1}}\|\nabla u\|_{L^{p_2}}\right).
\end{aligned}$$
Here $\frac{1}{p_1}+\frac{1}{p_2}=\frac{1}{2}$, $p_1<\infty$. Choose $p_1=\frac{4}{2-\alpha}$, $p_2=\frac{4}{\alpha}$, so that $s-\frac{2}{p_1}=s+\frac{\alpha}{2}-1$, $1-\frac{2}{p_2}=s+\frac{\alpha}{2}-1$. Then by Sobolev embedding,
$$\|\Lambda^su\|_{L^{p_1}}\|\nabla\rho_1\|_{L^{p_2}}+\|\Lambda^{s}\rho_1\|_{L^{p_1}}\|\nabla u\|_{L^{p_2}}\lesssim\|\Lambda^{s+\frac{\alpha}{2}}u\|_{L^2}\|\Lambda^{s+\frac{\alpha}{2}}\rho_1\|_{L^2}\lesssim\|\Lambda^{s+\frac{\alpha}{2}}\rho_1\|_{L^2}^2,$$ 
and we obtain (\ref{2.5}). On the other hand, for any $\alpha\in[1,2]$ and $s\geq1$, we choose $p_1=p_2=4$, so that
$$\begin{aligned}
&\|\Lambda^su\|_{L^{p_1}}\|\nabla\rho_1\|_{L^{p_2}}+\|\Lambda^{s}\rho_1\|_{L^{p_1}}\|\nabla u\|_{L^{p_2}}\\
\lesssim&
\|\Lambda^{s+\frac{1}{2}}u\|_{L^2}\|\Lambda^\frac{3}{2}\rho_1\|_{L^2}+\|\Lambda^{s+\frac{1}{2}}\rho_1\|_{L^2}\|\Lambda^\frac{3}{2}u\|_{L^2}\\
\lesssim&\|\Lambda^{s+\frac{\alpha}{2}}\rho_1\|_{L^2}^2+\|\Lambda^\frac{\alpha}{2}\rho_1\|_{L^2}^2.
\end{aligned}$$
Hence $$\left|\int_{\mathbb{R}^2}\Lambda^s\rho_1\Lambda^s(u\cdot\nabla\rho_1)dx\right|\lesssim\|\Lambda^s\rho_1\|_{L^2}\left(\|\Lambda^\frac{\alpha}{2}\rho_1\|^2_{L^2}+\|\Lambda^{s+\frac{\alpha}{2}}\rho_1\|_{L^2}^2\right).$$
Next, to prove (\ref{2.7})(\ref{2.8}), we write
$$\begin{aligned}
&\int_{\mathbb{R^2}}\Lambda^s\rho_1\Lambda^s(u\cdot\nabla\rho_0)dx=\int_{\mathbb{R}}(\Lambda_1^{s-\frac{\alpha}{2}}+\Lambda_2^{s-\frac{\alpha}{2}})^{-1}\Lambda^{2s}\rho_1(\Lambda_1^{s-\frac{\alpha}{2}}+\Lambda_2^{s-\frac{\alpha}{2}})(u\cdot\nabla\rho_0)dx.
\end{aligned}$$
By Lemma \ref{lem2.1}, $\|(\Lambda_1^{s-\frac{\alpha}{2}}+\Lambda_2^{s-\frac{\alpha}{2}})^{-1}\Lambda^{2s}\rho_1\|_{L^2}\lesssim\|\Lambda^{s+\frac{\alpha}{2}}\rho_1\|_{L^2}$ provided that $s\geq\frac{\alpha}{2}$, which is automatically satisfied if $s=2-\alpha$, $0<\alpha<1$ or $s\geq1$, $\alpha\leq2$. It suffices to bound $\|\Lambda_1^{s-\frac{\alpha}{2}}(u\cdot\nabla\rho_0)\|_{L^2}$ and $\|\Lambda_2^{s-\frac{\alpha}{2}}(u\cdot\nabla\rho_0)\|_{L^2}$. By Lemma \ref{lem2.1} again, for any $\alpha\in[0,2]$, $s\geq1$, we have
$$\begin{aligned}
\|\Lambda_1^{s-\frac{\alpha}{2}}(u\cdot\nabla\rho_0)\|_{L^2}=&\|\nabla\rho_0\cdot\Lambda_1^{s-\frac{\alpha}{2}}u\|_{L^2}
\leq\|\nabla\rho_0\|_{L^\infty_{x_2}}\|\Lambda_1^{s-\frac{\alpha}{2}}u\|_{L^2},
\end{aligned}$$
$$\begin{aligned}
\|\Lambda_2^{s-\frac{\alpha}{2}}(u\cdot\nabla\rho_0)\|_{L^2}
=&\left(\int_{\mathbb{R}}\|\Lambda_2^{s-\frac{\alpha}{2}}(u\cdot\nabla\rho_0)\|_{L^2_{x_2}}^2dx_1\right)^\frac{1}{2}\\
\lesssim&\left(\int_{\mathbb{R}}\|\nabla\rho_0\|^2_{L^\infty_{x_2}}\|\Lambda_2^{s-\frac{\alpha}{2}}u\|_{L^2}^2dx_1+\int_{\mathbb{R}}\|\nabla\tilde{\Lambda}^{s-\frac{\alpha}{2}}\rho_0\|_{L^p_{x_2}}^2\|u\|^2_{L^{\frac{2p}{p-2}}_{x_2}}dx_1\right)^\frac{1}{2}\\
\lesssim&\left(\|\nabla\rho_0\|_{L^\infty_{x_2}}+\|\nabla\tilde{\Lambda}^{s-\frac{\alpha}{2}}\rho_0\|_{L^p_{x_2}}\right)\left(\int_{\mathbb{R}}\|\Lambda_2^{s-\frac{\alpha}{2}}u\|_{L^2}^2dx_1+\int_{\mathbb{R}}\|u\|^2_{L^{\frac{2p}{p-2}}_{x_2}}dx_1\right)^\frac{1}{2}.
\end{aligned}$$
Using interpolation and Young's inequality, it holds
$$\|\Lambda_1^{s-\frac{\alpha}{2}}u\|_{L^2}\lesssim\|\Lambda_1^su\|_{L^2}^{1-\frac{\alpha}{2s}}\|u\|_{L^2}^\frac{\alpha}{2s}\lesssim\|u\|_{H^s},$$
$$\|\Lambda_2^{s-\frac{\alpha}{2}}u\|_{L^2_{x_2}}\lesssim\|\Lambda_2^{s}u\|_{L^2_{x_2}}^{1-\frac{\alpha}{2s}}\|u\|_{L^2_{x_2}}^\frac{\alpha}{2s}\lesssim\|u\|_{L^2_{x_2}}+\|\Lambda_2^su\|_{L^2_{x_2}},$$
$$\|u\|_{L^\frac{2p}{p-2}_{x_2}}\lesssim\|u\|_{L^2_{x_2}}+\|u\|_{L^\infty_{x_2}}\lesssim\|u\|_{L^2_{x_2}}+\|\Lambda_2^su\|_{L^2_{x_2}}.$$
Hence it follows that
$$\|\Lambda_1^{s-\frac{\alpha}{2}}(u\cdot\nabla\rho_0)\|_{L^2}+\|\Lambda_2^{s-\frac{\alpha}{2}}(u\cdot\nabla\rho_0)\|_{L^2}\lesssim\left(\|\nabla\rho_0\|_{L^\infty_{x_2}}+\|\nabla\tilde{\Lambda}^{s-\frac{\alpha}{2}}\rho_0\|_{L^p_{x_2}}\right)\|u\|_{H^s},$$
and we obtain (\ref{2.8}) by H\"older inequality.
Suppose now $\alpha\in(0,1)$, $s=2-\alpha$. For convenience, denote $\tilde{\rho}=(\Lambda^s_1+\Lambda_2^s)^{-1}\Lambda^s\rho_1$.
$$\begin{aligned}
\left|\int_{\mathbb{R}^2}\Lambda^s\rho_1\Lambda^s(u\cdot\nabla\rho_0)dx\right|=&\left|\int_{\mathbb{R}^2}\Lambda^{s}\tilde{\rho}(\Lambda_1^s+\Lambda_2^s)(u\cdot\nabla\rho_0)dx\right|\\
\lesssim&\int_{\mathbb{R}}\left\|\Lambda^{s}\tilde{\rho}\right\|_{L^\frac{2}{1-\alpha}_{x_2}}\left(\|\Lambda_1^s(u\cdot\nabla\rho_0)\|_{L^\frac{2}{1+\alpha}_{x_2}}+\|\Lambda_2^s(u\cdot\nabla\rho_0)\|_{L^\frac{2}{1+\alpha}_{x_2}}\right)dx_1.
\end{aligned}$$
Sobolev embedding gives that
$$\left\|\Lambda^{s}\tilde{\rho}\right\|_{L^\frac{2}{1-\alpha}_{x_2}}\lesssim\left\|\Lambda_2^{\frac{\alpha}{2}}\Lambda^s\tilde{\rho}\right\|_{L^2_{x_2}},$$
$$\begin{aligned}
\|\Lambda_1^s(u\cdot\nabla\rho_0)\|_{L^\frac{2}{1+\alpha}_{x_2}}=&\|\nabla\rho_0\cdot\Lambda_1^su\|_{L^\frac{2}{1+\alpha}_{x_2}}
\lesssim\|\nabla\rho_0\|_{L^\frac{1}{\alpha}_{x_2}}\|\Lambda_1^su\|_{L^\frac{2}{1-\alpha}_{x_2}}
\lesssim\|\nabla\rho_0\|_{L^\frac{1}{\alpha}_{x_2}}\|\Lambda_2^\frac{\alpha}{2}\Lambda_1^su\|_{L^2_{x_2}},
\end{aligned}$$
$$\begin{aligned}
\|\Lambda_2^s(u\cdot\nabla\rho_0)\|_{L^\frac{2}{1+\alpha}_{x_2}}\lesssim&\|\nabla\rho_0\|_{L^\frac{1}{\alpha}_{x_2}}\|\Lambda_2^su\|_{L^\frac{2}{1-\alpha}_{x_2}}+\|\nabla\tilde{\Lambda}^s\rho_0\|_{L^\frac{1}{\alpha}_{x_2}}\|u\|_{L^\frac{2}{1-\alpha}_{x_2}}\\
\lesssim&\|\nabla\rho_0\|_{L^\frac{1}{\alpha}_{x_2}}\|\Lambda_2^{s+\frac{\alpha}{2}}u\|_{L^2_{x_2}}+\|\nabla\tilde{\Lambda}^s\rho_0\|_{L^\frac{1}{\alpha}_{x_2}}\|\Lambda_2^\frac{\alpha}{2}u\|_{L_{x_2}^2}.
\end{aligned}$$
Therefore, 
$$\begin{aligned}
&\left|\int_{\mathbb{R}^2}\Lambda^s\rho_1\Lambda^s(u\cdot\nabla\rho_0)dx\right|\\
\lesssim&\int_{\mathbb{R}}\|\Lambda_2^\frac{\alpha}{2}\Lambda^s\tilde{\rho}\|_{L^2_{x_2}}\left(\|\nabla\rho_0\|_{L^\frac{1}{\alpha}_{x_2}}\|\Lambda_2^\frac{\alpha}{2}\Lambda_1^su\|_{L^2_{x_2}}+\|\nabla\rho_0\|_{L^\frac{1}{\alpha}_{x_2}}\|\Lambda_2^{s+\frac{\alpha}{2}}u\|_{L^2_{x_2}}+\|\nabla\tilde{\Lambda}^s\rho_0\|_{L^\frac{1}{\alpha}_{x_2}}\|\Lambda_2^\frac{\alpha}{2}u\|_{L_{x_2}^2}\right)dx_1\\
\lesssim&\|\Lambda_2^\frac{\alpha}{2}\Lambda^s\tilde{\rho}\|_{L^2}\left[\|\nabla\rho_0\|_{L^\frac{1}{\alpha}_{x_2}}\left(\|\Lambda_2^\frac{\alpha}{2}\Lambda_1^su\|_{L^2}+\|\Lambda_2^{s+\frac{\alpha}{2}}u\|_{L^2}\right)+\|\nabla\tilde{\Lambda}\rho_0\|_{L^\frac{1}{\alpha}_{x_2}}\|\Lambda_2^\frac{\alpha}{2}u\|_{L^2}\right].
\end{aligned}$$
Finally, we arrive at (\ref{2.7}) by the following inequalities
$$\|\Lambda_2^\frac{\alpha}{2}\Lambda^s\tilde{\rho}\|_{L^2}\leq\|\Lambda^{s+\frac{\alpha}{2}}\tilde{\rho}\|_{L^2}\simeq\|\Lambda^{s+\frac{\alpha}{2}}\rho_1\|_{L^2},$$
$$\|\Lambda_2^\frac{\alpha}{2}\Lambda_1^su\|_{L^2}+\|\Lambda_2^{s+\frac{\alpha}{2}}u\|_{L^2}\lesssim\|\Lambda^{s+\frac{\alpha}{2}}u\|_{L^2}\leq\|\Lambda^{s+\frac{\alpha}{2}}\rho_1\|_{L^2},$$
$$\|\Lambda_2^\frac{\alpha}{2}u\|_{L^2}\leq\|\Lambda^\frac{\alpha}{2}u\|_{L^2}\leq\|\Lambda^\frac{\alpha}{2}\rho_1\|_{L^2}.$$
\end{proof}
In the following we sketch a priori estimates that can be used to establish local well-posedness for (\ref{2.2}) in the small data case with $\alpha\in(0,1]$, $s=2-\alpha$ or $\alpha\in(1,2]$, $s\geq1$. A rigorous proof can be given with these estimates in a standard way of smooth mollifier approximations. We begin with the energy estimate of the solutions. From (\ref{2.2}), we have
$$\begin{aligned}
\frac{1}{2}\frac{d}{dt}\|\rho_1\|_{L^2}^2+\|\Lambda^\frac{\alpha}{2}\rho_1\|_{L^2}^2=-\int_{\mathbb{R}^2}\rho_1u\cdot\nabla\rho_0dx,
\end{aligned}$$
$$\begin{aligned}
\frac{1}{2}\frac{d}{dt}\|\Lambda^s\rho_1\|_{L^2}^2+\|\Lambda^{s+\frac{\alpha}{2}}\rho_1\|_{L^2}^2=-\int_{\mathbb{R}^2}\Lambda^s\rho_1\Lambda^s(u\cdot\nabla\rho_0)dx-\int_{\mathbb{R}^2}\Lambda^s\rho_1\Lambda^s(u\cdot\nabla\rho_1)dx.
\end{aligned}$$
By (\ref{2.4})(\ref{2.5})(\ref{2.6})(\ref{2.8}), we have for both cases
$$\begin{aligned}
&\frac{1}{2}\partial_t\|\rho_1\|_{H^s}^2+\|\Lambda^\frac{\alpha}{2}\rho_1\|_{L^2}^2+\|\Lambda^{s+\frac{\alpha}{2}}\rho_1\|_{L^2}^2\\
\lesssim&\|\rho_1\|_{L^2}^2\|\nabla\rho_0\|_{L^\infty}+\|\Lambda^s\rho_1\|_{L^2}\left(\|\Lambda^\frac{\alpha}{2}\rho_1\|_{L^2}^2+\|\Lambda^{s+\frac{\alpha}{2}}\rho_1\|_{L^2}^2\right)\\
&+\|\Lambda^{s+\frac{\alpha}{2}}\rho_1\|_{L^2}\|u\|_{H^s}\left(\|\nabla\rho_0\|_{L^\infty_{x_2}}+\|\nabla\tilde{\Lambda}^{s-\frac{\alpha}{2}}\rho_0\|_{L^p_{x_2}}\right).
\end{aligned}$$
Using Cauchy-Schwartz, it holds for $\delta<1$ and constant $C_\delta$ that
$$\begin{aligned}
&\|\Lambda^{s+\frac{\alpha}{2}}\rho_1\|_{L^2}\|u\|_{H^s}\left(\|\nabla\rho_0\|_{L^\infty_{x_2}}+\|\nabla\tilde{\Lambda}^{s-\frac{\alpha}{2}}\rho_0\|_{L^p_{x_2}}\right)\\
\lesssim&\delta\|\Lambda^{s+\frac{\alpha}{2}}\rho_1\|_{L^2}^2+C_\delta\|\rho_1\|^2_{H^s}\left(\|\nabla\rho_0\|_{L^\infty_{x_2}}+\|\nabla\tilde{\Lambda}^{s-\frac{\alpha}{2}}\rho_0\|_{L^p_{x_2}}\right)^2,
\end{aligned}$$
and thus
\begin{equation}\begin{aligned}
&\frac{1}{2}\partial_t\|\rho_1\|_{H^s}^2+(1-C_1\|\Lambda^s\rho_1\|_{L^2})\|\Lambda^\frac{\alpha}{2}\rho_1\|_{L^2}^2+(1-\delta-C_1\|\Lambda^s\rho_1\|_{L^2})\|\Lambda^{s+\frac{\alpha}{2}}\rho_1\|_{L^2}^2\\
\lesssim&\|\rho_1\|_{H^s}^2\left[\|\nabla\rho_0\|_{L^\infty}+C_\delta\left(\|\nabla\rho_0\|_{L^\infty_{x_2}}+\|\nabla\tilde{\Lambda}^{s-\frac{\alpha}{2}}\rho_0\|_{L^p_{x_2}}\right)^2\right].
\end{aligned}\label{2.9}\end{equation}
Let $\epsilon_1:=\frac{1}{2C_1}$ so that if we choose $\delta<\frac{1}{2}$, then
$$\begin{aligned}
\partial_t\|\rho_1\|_{H^s}^2\lesssim\|\rho_1\|_{H^s}^2\left[\|\nabla\rho_0\|_{L^\infty}+C_\delta\left(\|\nabla\rho_0\|_{L^\infty_{x_2}}+\|\nabla\tilde{\Lambda}^{s-\frac{\alpha}{2}}\rho_0\|_{L^p_{x_2}}\right)^2\right],
\end{aligned}$$
as long as the following condition holds:
\begin{equation}
\|\Lambda^s\rho_1\|_{L^2}\leq\epsilon_1.
\label{2.10}\end{equation}
Note that since $\rho_0$ is the solution to the heat equation (\ref{2.1}), we have
$$\|\nabla\rho_0\|_{L^\infty}\leq\|\nabla f\|_{L^\infty},\quad\|\Lambda^{s+1-\frac{\alpha}{2}}\rho_0\|_{L^p}\leq\|\Lambda^{s+1-\frac{\alpha}{2}}f\|_{L^p}.$$
By Gronwall's inequality, it follows
\begin{equation}\|\rho_1(t)\|_{H^s}\leq\|g\|_{H^s}\exp\left\{\int_0^t\left[\|\nabla\rho_0\|_{L^\infty}+C_\delta\left(\|\nabla\rho_0\|_{L^\infty_{x_2}}+\|\nabla\tilde{\Lambda}^{s-\frac{\alpha}{2}}\rho_0\|_{L^p_{x_2}}\right)^2\right]ds\right\}.
\label{2.11}\end{equation}
Such inequality can be used to show local well-posedness for (\ref{2.2}) provided $\|g\|_{H^s}$ is small enough.
\section{Proof of Theorem 1.1}
Now we prove Theorem \ref{thm1.1} for $\alpha$ in different regions.\\
\textit{Proof of Theorem \ref{thm1.1}}. Recall that we have the energy estimate
\begin{equation}\begin{aligned}
&\frac{1}{2}\frac{d}{dt}\|\rho_1\|_{H^s}^2+\|\Lambda^\frac{\alpha}{2}\rho_1\|_{L^2}^2+\|\Lambda^{s+\frac{\alpha}{2}}\rho_1\|_{L^2}^2\\
=&-\int_{\mathbb{R}^2}\rho_1u\cdot\nabla\rho_0dx-\int_{\mathbb{R}^2}\Lambda^s\rho_1\Lambda^s(u\cdot\nabla\rho_0)dx-\int_{\mathbb{R}^2}\Lambda^s\rho_1\Lambda^s(u\cdot\nabla\rho_1)dx.
\end{aligned}\label{3.1}\end{equation}
(\romannumeral1) The case $0<\alpha<1$. Suppose in this part $s=2-\alpha$. By (\ref{2.3})(\ref{2.5})(\ref{2.7}), we have
$$\begin{aligned}
&\left|\int_{\mathbb{R}^2}\rho_1u\cdot\nabla\rho_0dx\right|+\left|\int_{\mathbb{R}^2}\Lambda^s\rho_1\Lambda^s(u\cdot\nabla\rho_0)dx\right|+\left|\int_{\mathbb{R}^2}\Lambda^s\rho_1\Lambda^s(u\cdot\nabla\rho_1)dx\right|\\
\lesssim&\left(\|\nabla\rho_0\|_{L^\frac{1}{\alpha}}+\|\nabla\tilde{\Lambda}^s\rho_0\|_{L^\frac{1}{\alpha}}+\|\Lambda^s\rho_1\|_{L^2}\right)\left(\|\Lambda^\frac{\alpha}{2}\rho_1\|_{L^2}+\|\Lambda^{s+\frac{\alpha}{2}}\rho_1\|_{L^2}\right)\, 
\end{aligned}$$
so that for some universal constant $C_2$
\begin{equation}\frac{1}{2}\frac{d}{dt}\|\rho_1\|_{H^s}^2+C_2\left(C_2^{-1}-\|\nabla\rho_0\|_{L^\frac{1}{\alpha}}-\|\nabla\tilde{\Lambda}^s\rho_0\|_{L^\frac{1}{\alpha}}-\|\Lambda^s\rho_1\|_{L^2}\right)\left(\|\Lambda^\frac{\alpha}{2}\rho_1\|_{L^2}^2+\|\Lambda^{s+\frac{\alpha}{2}}\rho_1\|_{L^2}^2\right)\leq0.\label{3.2}\end{equation}
Hence we obtain the global solution once $\|\nabla\rho_0\|_{L^\frac{1}{\alpha}}+\|\nabla\Lambda^s\rho_0\|_{L^\frac{1}{\alpha}}+\|\Lambda^s\rho_1\|_{L^2}$ is small enough at some $t_1>0$. In fact, since $\rho_0$ is the solution of the heat equation (\ref{2.1}), the decay property of the heat equation (see \cite{MIAO2008461}, Lemma 3.1)  shows that
$$\|\nabla\rho_0\|_{L^{\frac{1}{\alpha}}(\mathbb{R})}\lesssim t^{-\frac{1}{\alpha}\left(\frac{1}{q}-\alpha\right)}\|\nabla f\|_{L^q(\mathbb{R})},$$
$$\|\nabla\tilde{\Lambda}^{s}\rho_0\|_{L^{\frac{1}{\alpha}}(\mathbb{R})}\lesssim t^{-\frac{s}{\alpha}-\frac{1}{\alpha}\left(\frac{1}{q}-\alpha\right)}\|\nabla f\|_{L^q(\mathbb{R})}.$$
We can choose $t_1>0$ large enough so that $\|\nabla\rho_0\|_{L^\frac{1}{\alpha}}+\|\nabla\Lambda^s\rho_0\|_{L^\frac{1}{\alpha}}\leq\epsilon_2:=\frac{1}{2C_2}$ for all $t\geq t_1$, then it suffices to require $\|\rho_1\|_{H^s}\leq\epsilon_2$  at $t=t_1$. To this end, we consider the interval $I:=\left\{t\in(0,t_1]\mid\|\rho_1(\tau)\|_{H^s}\leq\frac{1}{2}\min\left\{\epsilon_1,\epsilon_2\right\},\;\forall \tau\in[0,t]\right\}$. It follows from the local well-posedness theory for (\ref{2.2}) that $I$ is a nonempty closed subset of $(0,t_1]$ provided $\|g\|_{H^s}$ is small enough. We claim that $I$ is open in $(0,t_1]$, and thus $I=(0,t_1]$. Suppose $s_0\in I$, then by (\ref{2.11}) (view $\rho_1(s_0)$ as the initial data and propagate the control for $\|\rho_1\|_{H^s}$ to $t>s_0$) there exists $s_1>s_0$ such that $\|\rho_1(t)\|_{H^s}\leq\min\left\{\epsilon_1,\epsilon_2\right\}$ for all $t\in[0,s_1]$. Note that condition (\ref{2.10}) is satisfied on $[0,s_1]$. Now we choose $\epsilon$ small enough such that 
$$\epsilon\exp\left\{\int_0^{t_1}\left[\|\nabla\rho_0\|_{L^\infty}+C_\delta\left(\|\nabla\rho_0\|_{L^\infty_{x_2}}+\|\nabla\tilde{\Lambda}^{s-\frac{\alpha}{2}}\rho_0\|_{L^p_{x_2}}\right)^2\right]ds\right\}\leq\frac{1}{2}\min\left\{\epsilon_1,\epsilon_2\right\}.$$
It then follows from (\ref{2.11}) that $\|\rho_1(\tau)\|_{H^s}\leq\frac{1}{2}\min\left\{\epsilon_1,\epsilon_2\right\}$ for all $\tau\in[0,s_1]$. Consequently $s_1\in I$ and $(s_0,s_1]\subset I$, so $I$ is open. Now we have $\|\rho_1(t)\|_{H^s}\leq\frac{1}{2}\min\left\{\epsilon_1,\epsilon_2\right\}$ for all $t\in[0,t_1]$ and $\frac{d}{dt}\|\rho_1\|_{H^s}^2\leq0$ as long as $t\geq t_1$ and $\|\rho_1\|_{H^s}\leq\epsilon_2$. Therefore, using a standard continuity argument, we obtain the global solution of (\ref{2.2}).\\
(\romannumeral2) The case $1\leq\alpha<2$. Suppose now $s\geq1$ and $f\in L^q(\mathbb{R})$ with $q\in\left[1.\frac{1}{\alpha-1}\right)\cap[1,p]$. By the decay property of the heat equations, it holds
$$\|\nabla\rho_0\|_{L^\infty}\lesssim t^{-\frac{1}{\alpha}-\frac{1}{\alpha q}}\|f\|_{L^q},\quad \|\nabla\tilde{\Lambda}^{s-\frac{\alpha}{2}}\rho_0\|_{L^p}\lesssim t^{-\frac{1}{\alpha}\left(1+s-\frac{\alpha}{2}\right)-\frac{1}{\alpha}\left(\frac{1}{q}-\frac{1}{p}\right)}\|f\|_{L^q},$$
and 
$$\|\nabla\rho_0\|_{L^\infty}\leq\|\nabla f\|_{L^\infty},\quad \|\nabla\tilde{\Lambda}^{s-\frac{\alpha}{2}}\rho_0\|_{L^p}\leq\|\nabla\tilde{\Lambda}^{s-\frac{\alpha}{2}}f\|_{L^p}.$$ 
Noticing that $\frac{1}{\alpha}\left(1+\frac{1}{q}\right)>1$ and $\frac{1}{\alpha}\left(1+s-\frac{\alpha}{2}\right)+\frac{1}{\alpha}\left(\frac{1}{q}-\frac{1}{p}\right)>\frac{1}{\alpha}\left(\frac{\alpha}{2}+s-
\frac{1}{p}\right)>\frac{1}{2}$, these inequalities yield $$\int_0^\infty\|\nabla\rho_0(s)\|_{L^\infty}ds\lesssim\|\nabla f\|_{L^\infty}+\|f\|_{L^q},$$
$$\int_0^\infty\|\nabla\tilde{\Lambda}^{s-\frac{\alpha}{2}}\rho_0(s)\|_{L^p}^2ds\lesssim\|\tilde{\Lambda}^{1+s-\frac{\alpha}{2}}f\|_{L^p}^2+\|f\|_{L^q}^2.$$
Now we choose $\epsilon$ small enough so that
$$\epsilon\int_0^\infty\left[\|\nabla\rho_0(s)\|_{L^\infty}+C_{\delta}\left(\|\nabla\rho_0(s)\|_{L^2}^2+\|\nabla\tilde{\Lambda}^{s-\frac{\alpha}{2}}\rho_0(s)\|_{L^p_{x_2}}\right)^2\right]ds\leq\frac{1}{2}\epsilon_1.$$
Then (\ref{2.11}) gives that $\|\rho_1\|_{H^s}\leq\frac{1}{2}\epsilon_1$ as long as $\|\Lambda^s\rho_1\|_{L^2}\leq\epsilon_1$. Therefore, we obtain the global solution by a standard continuity argument. \\ \qed
\bibliography{IPM}
\end{document}